\documentclass[a4,12pt]{amsart}
\usepackage{amssymb,amstext,amsmath,amscd,amsthm,amsfonts,enumerate,latexsym}
\usepackage{color}
\usepackage[all]{xy}
\theoremstyle{plain}
\newtheorem{thm}{Theorem}[section]
\newtheorem*{thm*}{Theorem}
\newtheorem*{cor*}{Corollary}

\newtheorem{prop}[thm]{Proposition}
\newtheorem{lem}[thm]{Lemma}
\newtheorem{cor}[thm]{Corollary}

\newtheorem*{claim*}{Claim}

\theoremstyle{definition}

\newtheorem{ex}[thm]{Example}
\newtheorem{rem}[thm]{Remark}

\newtheorem{notation}[thm]{Notation}
\newtheorem{setup}[thm]{Setup}

\theoremstyle{remark}

\newtheorem{step}{Step}

\numberwithin{equation}{thm}

\def\Z{\mathbb{Z}}

\def\pd{\operatorname{pd}}
\def\Ext{\operatorname{Ext}}

\def\Ker{\operatorname{Ker}}

\def\Hom{\operatorname{Hom}}

\def\Mod{\mathrm{Mod}}

\def\rank{\mathrm{rank}}
\def\a{\mathrm a}

\def\e{\mathrm{e}}
\def\m{\mathfrak m}
\def\n{\mathfrak n}

\def\q{\mathfrak q}

\def\H{\mathrm{H}}

\newcommand{\rmK}{\mathrm{K}}

\newcommand{\calR}{\mathcal{R}}

\newcommand{\calX}{\mathcal{X}}

\newcommand{\fka}{\mathfrak{a}}
\newcommand{\fkb}{\mathfrak{b}}
\newcommand{\fkc}{\mathfrak{c}}

\newcommand{\fkm}{\mathfrak{m}}
\newcommand{\fkn}{\mathfrak{n}}

\newcommand{\fkp}{\mathfrak{p}}

\newcommand{\fkM}{\mathfrak{M}}

\newcommand{\mapright}[1]{%
\smash{\mathop{%
\hbox to 1cm{\rightarrowfill}}\limits^{#1}}}

\newcommand{\mapleft}[1]{%
\smash{\mathop{%
\hbox to 1cm{\leftarrowfill}}\limits_{#1}}}

\def\ann{\operatorname{Ann}}

\def\Spec{\operatorname{Spec}}
\def\Syz{\mathrm{Syz}}

\def\gr{\mbox{\rm gr}}

\def\Y{{\mathcal Y}}

\def\ol{\overline}

\def\rhom{\operatorname{\mathbf{R}Hom}}
\def\gdim{\operatorname{Gdim}}

\begin{document}
\setlength{\baselineskip}{15pt}
\title{Ulrich ideals and almost Gorenstein rings}
\author{Shiro Goto}
\address{Department of Mathematics, School of Science and Technology, Meiji University, 1-1-1 Higashi-mita, Tama-ku, Kawasaki 214-8571, Japan}
\email{goto@math.meiji.ac.jp}
\author{Ryo Takahashi}
\address{Graduate School of Mathematics, Nagoya University, Furocho, Chikusa-ku, Nagoya 464-8602, Japan}
\urladdr{http://www.math.nagoya-u.ac.jp/~takahashi/}
\email{takahashi@math.nagoya-u.ac.jp}
\author{Naoki Taniguchi}
\address{Department of Mathematics, School of Science and Technology, Meiji University, 1-1-1 Higashi-mita, Tama-ku, Kawasaki 214-8571, Japan}
\email{taniguti@math.meiji.ac.jp}
\thanks{The first author was partially supported by JSPS Grant-in-Aid for Scientific Research 25400051. The second author was partially supported by JSPS Grant-in-Aid for Scientific Research 25400038. The third author was partially supported by Grant-in-Aid for JSPS Fellows 26-126 and by JSPS Research Fellow. }
\begin{abstract} 
The structure of the complex $\rhom_R(R/I,R)$ is explored for an Ulrich ideal $I$ in a Cohen--Macaulay local ring $R$.
As a consequence, it is proved that in a one-dimensional almost Gorenstein but non-Gorenstein local ring, the only possible Ulrich ideal is the maximal ideal.
It is also studied when Ulrich ideals have the same minimal number of generators.
\end{abstract}
\maketitle
\section{Introduction}
This paper studies Ulrich ideals of Cohen--Macaulay local rings and almost Gorenstein local rings.

Ulrich ideals are newcomers.
They were introduced by \cite{GOTWY1} in 2014.
Typical examples of Ulrich ideals are the maximal ideal of a Cohen--Macaulay local ring with minimal multiplicity.
The syzygy modules of Ulrich ideals are known to be well-behaved \cite {GOTWY1}.
We refer the reader to \cite{GOTWY1} for a basic theory of Ulrich ideals and \cite{GOTWY2} for the results about the ubiquity of Ulrich ideals of two-dimensional rational singularities and the representation-theoretic aspects of Ulrich ideals.

Almost Gorenstein rings are also newcomers.
They form a class of Cohen--Macaulay rings, which are not necessarily Gorenstein but still good, hopefully next to the Gorenstein rings.
The notion of almost Gorenstein local rings dates back to the article \cite{BF} of Barucci and Fr\"oberg in 1997.
They introduced almost Gorenstein rings in the case where the local rings are of dimension one and analytically unramified.
We refer the reader to \cite{BF} for a well-developed theory of almost symmetric numerical semigroups.
The notion of almost Gorenstein local rings in the present  paper is, however, based on the definition given by the authors \cite{GTT} in 2015 for Cohen--Macaulay local rings of arbitrary  dimension.
See \cite{GMP} for a basic theory of almost Gorenstein local rings of dimension one which might be analytically ramified.

One of the purposes of this paper is to clarify the structure of Ulrich ideals of almost Gorenstein local rings.
The motivation for the research comes from a recent result of Kei-ichi Watanabe, which asserts that non-Gorenstein almost Gorenstein numerical semigroup rings possess no Ulrich monomial ideals except the maximal ideal.
This result essentially says that there should be some restriction of the distribution of Ulrich ideals of an almost Gorenstein but non-Gorenstein local ring.
Our research started from the attempt to understand this phenomenon.
Along the way, we recognized that his result holds true for every one-dimensional almost Gorenstein non-Gorenstein local ring, and finally reached new knowledge about the behavior of Ulrich ideals, which is reported in this paper.

Let us state the results of this paper, explaining how this paper is organized.
In Section \ref{rhom} we shall prove the following structure theorem of the complex $\rhom_R(R/I,R)$ for an Ulrich ideal $I$.

\begin{thm}
Let $R$ be a Cohen--Macaulay local ring of dimension $d\ge0$.
Let $I$ be a non-parameter Ulrich ideal of $R$ containing a parameter ideal of $R$ as a reduction.
Denote by $\nu(I)$ the minimal number of generators of $I$, and put $t=\nu(I)-d$.
Then there is an isomorphism
$$
\rhom_R(R/I,R)\cong\bigoplus_{i\in\Z}(R/I)^{\oplus u_i}[-i]
$$
in the derived category of $R$, where
$$
u_i=
\begin{cases}
0 & (i<d),\\
t & (i=d),\\
(t^2-1)t^{i-d-1} & (i>d).
\end{cases}
$$
In particular, one has $\Ext_R^i(R/I,R)\cong(R/I)^{\oplus u_i}$ for each integer $i$.
\end{thm}

This theorem actually yields a lot of consequences and applications.
Let us state some of them.
The Bass numbers of $R$ are described in terms of those of $R/I$ and the $u_i$, which recovers a result in \cite{GOTWY1}.
Finiteness of the G-dimension of $I$ is characterized in terms of $\nu(I)$, which implies that if $R$ is {\em G-regular} in the sense of \cite{greg} (e.g., $R$ is a non-Gorenstein ring with minimal multiplicity, or is a non-Gorenstein almost Gorenstein ring), then one must have $\nu(I)\ge d+2$.
For a non-Gorenstein almost Gorenstein ring with prime Cohen--Macaulay type, all the Ulrich ideals have the same minimal number of generators.
For every one-dimensional non-Gorenstein almost Gorenstein local ring the only non-parameter Ulrich ideal is the maximal ideal.
This recovers the result of Watanabe mentioned above, and thus our original aim of the research stated above is achieved.

Now we naturally get interested in whether or not the minimal numbers of generators of Ulrich ideals of an almost Gorenstein non-Gorenstein local ring are always constant.
We will explore this in Section \ref{const} to obtain some supporting evidence for the affirmativity.
By the way, it turns out to be no longer true if the base local ring is not almost Gorenstein.
In Section \ref{method} we will give a method of constructing Ulrich ideals which possesses different numbers of generators.

\begin{notation}
In what follows, unless otherwise specified, $R$ stands for a $d$-dimensional Cohen--Macaulay local ring with maximal ideal $\fkm$ and residue field $k$.
For a finitely generated $R$-module $M$, denote by $\ell_R(M)$, $\nu_R(M)$, $r_R(M)$ and $\e_\m^0(M)$ the length of $M$, the minimal number of generators of $M$, the Cohen--Macaulay type of $M$ and the multiplicity of $M$ with respect to $\fkm$.
Let $v(R)$ denote the embedding dimension of $R$, i.e., $v(R)=\nu_R(\m)$.
For each integer $i$ we denote by $\mu_i(R)$ the $i$-th Bass number of $R$, namely, $\mu_i(R)=\dim_k\Ext_R^i(k,R)$.
Note that $\mu_d(R)=r(R)$.
The subscript indicating the base ring is often omitted.
\end{notation}

\section{The structure of $\rhom_R(R/I,R)$ for an Ulrich ideal $I$}\label{rhom}

In this section, we establish a structure theorem of $\rhom_R(R/I,R)$ for an Ulrich ideal $I$ of a Cohen--Macaulay local ring, and derive from it a lot of consequences and applications.
First of all, we fix our notation and assumptions on which all the results in this section are based.

\begin{setup}
Throughout this section, let $I$ be a non-parameter $\m$-primary ideal of $R$ containing a parameter ideal $Q$ of $R$ as a reduction.
Suppose that $I$ is an {\em Ulrich ideal}, that is, $I^2=QI$ and $I/I^2$ is $R/I$-free.
Put $t=\nu(I)-d>0$ and
$$
u_i=
\begin{cases}
0 & (i<d),\\
t & (i=d),\\
(t^2-1)t^{i-d-1} & (i>d).
\end{cases}
$$
\end{setup}

\begin{rem}
(1) The condition that $I$ contains a parameter ideal $Q$ of $R$ as a reduction is automatically satisfied if $k$ is infinite.\\
(2) The condition $I^2=QI$ is independent of the choice of minimal reductions $Q$ of $I$.
\end{rem}

The following is the main result of this section.

\begin{thm}\label{main}
There is an isomorphism
$$
\rhom_R(R/I,R)\cong\bigoplus_{i\in\Z}(R/I)^{\oplus u_i}[-i]
$$
in the derived category of $R$.
Hence for each integer $i$ one has an isomorphism
$$
\Ext_R^i(R/I,R)\cong(R/I)^{\oplus u_i}
$$
of $R$-modules.
In particular, $\Ext_R^i(R/I,R)$ is a free $R/I$-module.
\end{thm}

\begin{proof}
Let us first show that $\Ext_R^i(R/I,R)\cong(R/I)^{\oplus u_i}$ for each $i$.
We do it by making three steps.

\begin{step}\label{1}
As $I$ is an $\m$-primary ideal, $R/I$ has finite length as an $R$-module.
Hence we have $\Ext_R^{<d}(R/I,R)=0$.
\end{step}

\begin{step}\label{2}
There is a natural exact sequence $0 \to I/Q \xrightarrow{f} R/Q \xrightarrow{g} R/I \to 0$, which induces an exact sequence
$$
\begin{array}{cccccc}
&\Ext_R^d(R/I,R)& \longrightarrow &\Ext_R^d(R/Q,R)& \longrightarrow &\Ext_R^d(I/Q,R)\\
\longrightarrow &\Ext_R^{d+1}(R/I,R)& \longrightarrow &\Ext_R^{d+1}(R/Q,R).
\end{array}
$$
Since $Q$ is generated by an $R$-sequence, we have $\Ext_R^{d+1}(R/Q,R)=0$.
There is a commutative diagram
$$
\xymatrix{
& \Ext_R^d(R/I,R)\ar[r]\ar[d]^\cong & \Ext_R^d(R/Q,R)\ar[r]\ar[d]^\cong & \Ext_R^d(I/Q,R)\ar[d]^\cong \\
0\ar[r] & \Hom_{R/Q}(R/I,R/Q)\ar[r]^{g^*}\ar[d]^\cong & \Hom_{R/Q}(R/Q,R/Q)\ar[r]^{f^*}\ar[d]^\cong & \Hom_{R/Q}(I/Q,R/Q) \\
0\ar[r] & (Q:I)/Q\ar[r]^{h} & R/Q
}
$$
with exact rows, where the vertical maps are natural isomorphisms and $h$ is an inclusion map.
Thus we get an exact sequence
$$
0 \to (Q:I)/Q \xrightarrow{h} R/Q \to \Hom_{R/Q}(I/Q,R/Q) \to \Ext_R^{d+1}(R/I,R) \to 0.
$$
Note here that $I/Q\cong(R/I)^{\oplus t}$ and $Q:I=I$ hold; see \cite[Lemma 2.3 and Corollary 2.6]{GOTWY1}.
Hence $\Ext_R^d(R/I,R)\cong(Q:I)/Q=I/Q\cong(R/I)^{\oplus t}$.
We have isomorphisms $\Hom_{R/Q}(I/Q,R/Q)\cong\Hom_{R/Q}(R/I,R/Q)^{\oplus t}\cong(I/Q)^{\oplus t}\cong(R/I)^{\oplus t^2}$, and therefore we obtain an exact sequence
$$
0 \to R/I \to (R/I)^{\oplus t^2} \to \Ext_R^{d+1}(R/I,R) \to0.
$$
This exact sequence especially says that $\Ext_R^{d+1}(R/I,R)$ has finite projective dimension as an $R/I$-module.
Since $R/I$ is an Artinian ring, it must be free, and we see that $\Ext_R^{d+1}(R/I,R)\cong(R/I)^{\oplus t^2-1}$.
\end{step}

\begin{step}\label{3}
It follows from \cite[Corollary 7.4]{GOTWY1} that $\Syz_R^i(R/I)\cong\Syz_R^d(R/I)^{\oplus t^{i-d}}$ for each $i\ge d$.
Hence we have
\begin{align*}
\Ext_R^{i+1}(R/I,R)
&\cong\Ext_R^1(\Syz_R^i(R/I),R)\cong\Ext_R^1(\Syz_R^d(R/I)^{\oplus t^{i-d}},R)\\
&\cong\Ext_R^{d+1}(R/I,R)^{\oplus t^{i-d}}\cong(R/I)^{\oplus(t^2-1)t^{i-d}}
\end{align*}
for all $i\ge d$.
\end{step}

Combining the observations in Steps \ref{1}, \ref{2} and \ref{3} yields that $\Ext_R^i(R/I,R)\cong(R/I)^{\oplus u_i}$ for all $i\in\Z$.

Take an injective resolution $E$ of $R$.
Then note that $C:=\Hom_R(R/I,E)$ is a complex of $R/I$-modules with $\H^i(C)\cong\Ext_R^i(R/I,R)\cong(R/I)^{\oplus u_i}$ for every $i\in\Z$.
Hence each homology $\H^i(C)$ is a projective $R/I$-module.
Applying \cite[Lemma 3.1]{ddc} to the abelian category $\Mod R/I$, the category of (all) $R/I$-modules, we obtain isomorphisms $\rhom_R(R/I,R)\cong C\cong\bigoplus_{i\in\Z}\H^i(C)[-i]\cong\bigoplus_{i\in\Z}(R/I)^{\oplus u_i}[-i]$ in the derived category of $R/I$.
This completes the proof of the theorem.
\end{proof}

The remainder of this section is devoted to producing consequences and applications of the above theorem.
First, we investigate vanishing of Ext modules.

\begin{cor}\label{extbase}
Let $M$ be a (possibly infinitely generated) $R/I$-module.
There is an isomorphism
$$
\Ext_R^i(M,R)\cong\bigoplus_{j\in\Z}\Ext_{R/I}^{i-j}(M,R/I)^{\oplus u_j}
$$
for each integer $i$.
In particular, if $\Ext_R^{\gg0}(M,R)=0$, then $\Ext_{R/I}^{\gg0}(M,R/I)=0$.
\end{cor}

\begin{proof}
There are isomorphisms
$$
\rhom_R(M,R)\cong\rhom_{R/I}(M,\rhom_R(R/I,R))\cong\bigoplus_{j\in\Z}\rhom_{R/I}(M,R/I)^{\oplus u_j}[-j],
$$
where the first isomorphism holds by adjointness (see \cite[(A.4.21)]{C}) and the second one follows from Theorem \ref{main}.
Taking the $i$th homologies, we get an isomorphism $\Ext_R^i(M,R)\cong\bigoplus_{j\in\Z}\Ext_{R/I}^{i-j}(M,R/I)^{\oplus u_j}$ for all integers $i$.
Since $u_d=t>0$, the module $\Ext_{R/I}^{i-d}(M,R/I)$ is a direct summand of $\Ext_R^i(M,R)$.
Therefore, when $\Ext_R^{\gg0}(M,R)=0$, one has $\Ext_{R/I}^{\gg0}(M,R/I)=0$.
\end{proof}

Now we can calculate the Bass numbers of $R$ in terms of those of $R/I$.

\begin{thm}\label{bass}
There are equalities
$$
\mu_i(R)=\sum_{j\in\Z}u_j\mu_{i-j}(R/I)=
\begin{cases}
0 & (i<d),\\
\sum_{j=d}^iu_j\mu_{i-j}(R/I) & (i\ge d).
\end{cases}
$$
In particular, one has
$$
t\cdot r(R/I)=r(I/Q)=r(R).
$$
\end{thm}

\begin{proof}
Applying Corollary \ref{extbase} to the $R/I$-module $k$ gives rise to an isomorphism $\Ext_R^i(k,R)\cong\bigoplus_{j\in\Z}\Ext_{R/I}^{i-j}(k,R/I)^{\oplus u_j}$ for each integer $i$.
Compaing the $k$-dimension of both sides, we get $\mu_i(R)=\sum_{j\in\Z}u_j\mu_{i-j}(R/I)$.
Hence we have $r(R)=\mu_d(R)=u_d\mu_0(R/I)=t\cdot r(R/I)=r(I/Q)$, where the last equality comes from the isomorphism $I/Q\cong(R/I)^{\oplus t}$ (see \cite[Lemma 2.3]{GOTWY1}).
Thus all the assertions follow.
\end{proof}

The above theorem recovers a result of Goto, Ozeki, Takahashi, Watanabe and Yoshida.

\begin{cor}\cite[Corollary 2.6(b)]{GOTWY1}
The following are equivalent.
\begin{enumerate}[\rm(1)]
\item
$R$ is Gorenstein.
\item
$R/I$ is Gorenstein and $\nu(I)=d+1$.
\end{enumerate}
\end{cor}

\begin{proof}
Theorem \ref{bass} implies $t\cdot r(R/I)=r(R)$.
Hence $r(R)=1$ if and only if $t=r(R/I)=1$.
This shows the assertion.
\end{proof}

To state our next results, let us recall some notions.

A {\em totally reflexive} $R$-module is by definition a finitely generated reflexive $R$-module $G$ such that $\Ext_R^{>0}(G,R)=0=\Ext_R^{>0}(\Hom_R(G,R),R)$.
Note that every finitely generated free $R$-module is totally reflexive.
The {\em Gorenstein dimension} ({\em G-dimension} for short) of a finitely generated $R$-module $M$, denoted by $\gdim_RM$, is defined as the infimum of integers $n\ge0$ such that there exists an exact sequence
$$
0 \to G_n \to G_{n-1} \to \cdots \to G_0 \to M \to 0
$$
of $R$-modules with each $G_i$ totally reflexive.

A Noetherian local ring $R$ is called {\em G-regular} if every totally reflexive $R$-module is free.
This is equivalent to saying that the equality $\gdim_RM=\pd_RM$ holds for all finitely generated $R$-modules $M$.

\begin{rem}\label{rem}
The following local rings are G-regular.
\begin{itemize}
\item
Regular local rings.
\item
Non-Gorenstein Cohen--Macaulay local rings with minimal multiplicity.
\item
Non-Gorenstein almost Gorenstein local rings.
\end{itemize}
For the proofs, we refer to \cite[Proposition 1.8]{greg}, \cite[Examples 3.5]{AM} (see also \cite[Corollary 2.5]{Y}) and \cite[Corollary 4.5]{GTT}, respectively.
\end{rem}

Suppose that $R$ admits a canonical module $\rmK_R$.
We say that $R$ is {\em almost Gorenstein} if there exists an exact sequence
$$
0 \to R \to \rmK_R \to C \to 0
$$
of $R$-modules such that $C$ is an {\em Ulrich $R$-module}, i.e., $C$ is a Cohen--Macaulay $R$-module (of dimension $d-1$) with $\e_\m^0(C)=\nu_R(C)$.

Using our Theorem \ref{main}, we establish a characterization of finiteness of the G-dimension of $R/I$ in terms of the minimal number of generator of $I$.

\begin{thm}\label{t=1}
One has
$$
\nu(I)=d+1\ \Longleftrightarrow\ \gdim_RR/I<\infty.
$$
In particular, if $R$ is G-regular, then $\nu(I)\ge d+2$.
\end{thm}

\begin{proof}
As to the first assertion, it suffices to show that $t=1$ if and only if $R/I$ has finite G-dimension as an $R$-module.

The `if' part:
As $R/I$ has depth $0$, it has G-dimension $d$ by \cite[(1.4.8)]{C}, and hence $\Ext_R^{>d}(R/I,R)=0$ by \cite[(1.2.7)]{C}.
It follows from Theorem \ref{main} that $u_i=0$ for all $i>d$.
In particular, we have $t^2-1=u_{d+1}=0$, which implies $t=1$.

The `only if' part:
By Theorem \ref{main} we have $\rhom_R(R/I,R)\cong R/I[-d]$.
It is observed from this that the homothety morphism $R/I\to\rhom_R(\rhom_R(R/I,R),R)$ is an isomorphism.
It follows from \cite[(2.2.3)]{C} that the $R$-module $R/I$ has finite G-dimension.

Thus the first assertion of the theorem follows.
As for the second assertion, suppose that $t=1$.
Then $R/I$ has finite G-dimension, and so does $I$ by \cite[(1.2.9)]{C}.
Since $R$ is G-regular, $I$ has finite projective dimension.
As $I$ is an Ulrich ideal, $I/I^2$ is a free $R/I$-module.
Hence we see from \cite[Theorem 2.2.8]{BH} that $I$ is generated by an $R$-sequence, which contradicts the assumption that $I$ is a non-parameter $\m$-primary ideal.
Therefore we have $t\ge2$, which means $\nu(I)\ge d+2$.
\end{proof}

As a consequence of the above theorem, we have a characterization of Gorenstein local rings.

\begin{cor}\label{4}
The following are equivalent.
\begin{enumerate}[\rm(1)]
\item
$R$ is Gorenstein.
\item
There is an Ulrich ideal $I$ of $R$ with finite G-dimension such that $R/I$ is Gorenstein.
\end{enumerate}
\end{cor}

\begin{proof}
(1) $\Rightarrow$ (2):
Any parameter ideal $I$ of $R$ is such an ideal as in the condition (2).

(2) $\Rightarrow$ (1):
It is trivial if $I$ is a parameter ideal, so suppose that $I$ is not so.
The Gorensteinness of $R/I$ implies $\mu_0(R/I)=1$ and $\mu_{>0}(R/I)=0$, and hence $\mu_i(R)=u_i$ for all $i\ge d$ by Theorem \ref{bass}.
Since $R/I$ has finite G-dimension, we have $t=1$ by Theorem \ref{t=1}, whence $u_d=1$ and $u_{>d}=0$.
Thus we get $\mu_d(R)=1$ and $\mu_{>d}(R)=0$, which shows that $R$ is Gorenstein.
\end{proof}

\begin{rem}
Corollary \ref{4} is a special case of \cite[Theorem 2.3]{psit}, which implies that a (not necessarily Cohen--Macaulay) local ring $R$ is Gorenstein if and only if it possesses a (not necessarily Ulrich) ideal $I$ of finite G-dimension such that $R/I$ is Gorenstein.
\end{rem}

Using our theorems, we observe that the minimal numbers of generators of Ulrich ideals are constant for certain almost Gorenstein rings.

\begin{cor}\label{2.4a}
Let $R$ be a non-Gorenstein almost Gorenstein local ring such that $r(R)$ is a prime number.
Then $R/I$ is a Gorenstein ring and $\nu(I)  = r(R) + d$.
\end{cor}

\begin{proof}
It follows from Theorem \ref{bass} that $t\cdot r(R/I)=r(R)$.
Since $t>1$ by Theorem \ref{t=1} and Remark \ref{rem}, we must have that $r(R/I)=1$ and $r(R)=t=\nu(I)-d$.
\end{proof}

\begin{cor}\label{referee}
Let $R$ be a two-dimensional rational singularity. Then $\nu(I)  = r(R) + 2$.
\end{cor}

\begin{proof}
Because $r(R/I) =1$ by \cite[Corollary 6.5]{GOTWY2}, the equality follows from Theorem \ref{bass}.
\end{proof}

The following corollary is another consequence of Theorem \ref{main}.
Note that such an exact sequence as in the corollary exists for every almost Gorenstein ring.

\begin{cor}\label{2.5}
Suppose that $R$ admits a canonical module $\rmK_R$, and that there is an exact sequence $0 \to R \to \rmK_R \to C \to 0$ of $R$-modules.
If $\nu(I)\ge d+2$, then $\ann_RC\subseteq I$.
\end{cor}

\begin{proof}
We set $\fka = \ann_RC$ and $M = \Syz_R^d(R/I)$. Then $M$ is a maximal Cohen--Macaulay $R$-module.
Hence $\Ext_R^{>0}(M,\rmK_R) = 0$, and in particular there is a surjection $\Hom_R(M,C) \twoheadrightarrow \Ext_R^1(M,R)$.
Since $\Ext_R^1(M,R) \cong \Ext_R^{d+1}(R/I, R)$ and $t>1$, the ideal $\fka$ annihilates $R/I$ by Theorem \ref{main}, whence $I$ contains $\fka$. 
\end{proof}

Now we state the last theorem in this section, whose first assertion is proved by Kei-ichi Watanabe in the case where $R$ is a numerical semigroup ring over a field and all ideals considered are monomial. 

\begin{thm}\label{2.6}
Let $R$ be a non-Gorenstein local ring of dimension $d$.
Assume that $R$ is almost Gorenstein, that is, there exists an exact sequence $0 \to R \to \rmK_R \to C \to 0$ such that $C$ is Ulrich.
\begin{enumerate}[\rm(1)]
\item
If $d=1$, then $I = \fkm$.
\item
Suppose that $k$ is infinite.
If $\m C = IC$, then $I=\m$.
\end{enumerate}
\end{thm}

\begin{proof}
(1) As $C\ne0=\m C$, we have $\nu(I)\ge d+2$ by Theorem \ref{t=1} and Remark \ref{rem}.
Hence $I = \m$ by Corollary  \ref{2.5}.

(2) We may assume by (1) that $d > 1$ and that our assertion holds true for $d-1$.
We set $\fka = \ann_RC$ and $S = R/\fka$.
Then $\m S$ is integral over $I S$, because $\m C = IC$.
Therefore, without loss of generality, we may assume that $a = a_1$ (a part of a minimal basis of a reduction $Q = (a_1, a_2, \ldots, a_d)$ of $I$) is a superficial element of $C$ with respect to $\m$.
Let $\ol{R} = R/(a)$, $\ol{I} = I/(a)$, and  $\ol{C} = C/aC$.
Then $\ol{R}$ is a non-Gorenstein  almost Gorenstein ring, $\ol{C}$ is an Ulrich $\ol{R}$-module,  and we have an exact sequence $0 \to \ol{R} \to \rmK_{\ol{R}} \to \ol{C} \to 0$ of $\ol{R}$-modules (\cite[Proof of Theorem 3.7 (2)]{GTT}), because $\rmK_{\ol{R}} \cong \rmK_R/a\rmK_R$ (\cite[Korollar 6.3]{HK}).
Consequently, since $\m \ol{C} = I \ol{C}$, we get $I \ol{R} = \m\ol{R}$ by the hypothesis of induction, so that $I = \m$.
\end{proof}


\section{The expected core of Ulrich ideals}\label{const}
In this section let $(R,\m,k)$ be a $d$-dimensional Cohen-Macaulay local ring with canonical module $\rmK_R$. We denote by $\calX_R$ the set of non-parameter Ulrich ideals of $R$. Let $$\fka = \sum_{f \in K_R~\text{such~that}~0:_Rf  =0}\left[Rf:_R\rmK_R\right], $$ which is the {\it expected core} of Ulrich ideals in the case where $R$ is an almost Gorenstein but non-Gorenstein ring. In fact we have the following.

\begin{thm}\label{3.4}
Suppose that $R$ is a non-Gorenstein almost Gorenstein local domain. Then  the following assertions hold true.
\begin{enumerate}[$(1)$]
\item If $I \in \calX_R$, then $\fka \subseteq I$. 
\item Suppose that $R_\fkp$ is a Gorenstein ring for every $\fkp \in \Spec R \setminus \{{\fkm}\}$. Then $\sqrt{\fka}=\m$.
\item Suppose that  $\dim R = 2$ and $r(R) = 2$. If $R_\fkp$ is a Gorenstein ring for every $\fkp \in \Spec R \setminus \{{\fkm}\}$, then $\calX_R$ is a finite set and every $I \in \calX_R$ is minimally generated by four elements.
\end{enumerate}
\end{thm}

\begin{proof} (1) For each $f \in \rmK_R$ such that $0:_Rf = 0$ we have an exact sequence $$0 \to R \overset{\varphi}{\longrightarrow} \rmK_R \to C \to 0$$ with $\varphi (1) = f$ and applying Corollary \ref{2.5} to the sequence, we get $\fka \subseteq I$ by Theorem \ref{t=1}.

(2) Let $\fkp \in \Spec R \setminus \{\m\}$. Then $[\rmK_R]_\fkp =\rmK_{R_\fkp} \cong R_\fkp$, since $R_\fkp$ is a Gorenstein ring. Choose an element $f \in \rmK_R$ so that $[\rmK_R]_\fkp = R_\fkp \frac{f}{1}$. Then $0:_Rf = 0$ and $Rf:_R\rmK_R \not\subseteq \fkp$. Hence  $\fka \not\subseteq \fkp$ and therefore $\sqrt{\fka} = \m$.

(3) We have an exact sequence
$$0 \to R \to \rmK_R \to C \to 0$$ of $R$-modules such that $C = R/\fkp$ is a DVR (\cite[Corollary 3.10]{GTT}). Then $\fkp \subseteq \fka$ by the definition of $\fka$.  Since $\fka \ne \fkp$ by assertion (2), we have $\fka = \fkp + x^nR$ for some $n > 0$ where $x \in \m$ such that $\m = \fkp + (x)$. Let $I \in \calX_R$. Then because $\fka \subseteq I$ by assertion (1), we get $I = \fkp + x^\ell R$ with $1 \le \ell \le n$. Hence the set $\calX_R$ is finite. By Corollary \ref{2.4a} every $I \in \calX_R$ is minimally generated by four elements.
\end{proof}

\begin{cor}
Let $R$ be a two-dimensional normal local ring. Assume that $R$ is a non-Gorenstein almost Gorenstein ring with $\mathrm{r}(R) = 2$. Then $\calX_R$ is a finite set and every $I \in \calX_R$ is minimally generated by four elements.
\end{cor}

\begin{rem}
We know no examples of non-Gorenstein almost Greenstein local rings in which Ulrich ideals do not possess a common number of generators.
\end{rem}

We explore a few examples. Let $S = k[X,Y,Z,W]$ be the polynomial ring over a field $k$. Let $n \ge 1$ be an integer and consider the matrix $\Bbb M =( \begin{smallmatrix}
X^n&Y&Z\\
Y&Z&W
\end{smallmatrix})$. We set $T = S/\mathrm{I}_2(\Bbb M)$ where $\mathrm{I}_2(\Bbb M)$ denotes the ideal of $S$ generated by two by two minors of $\Bbb M$. Let $x,y,z,w$ denote the images of $X,Y,Z,W$ in $T$ respectively. We set $R = T_M$, where $M = (x,y,z,w)T$.

\begin{thm}\label{3.5} We have the following. 
\begin{enumerate}
\item[$(1)$] $R$ is a non-Gorenstein almost Gorenstein local integral domain with $ r(R) = 2$. 
\item[$(2)$] $\calX_R = \{(x^\ell, y, z, w)R \mid 1 \le \ell \le n\}$. 
\item[$(3)$] $R$ is a normal ring if and only if $n = 1$.
\end{enumerate}
\end{thm}

\begin{proof} We regard $S$ as a $\Bbb Z$-graded ring so that $\deg X = 1, \deg Y = n+1, \deg Z = n+2$, and $\deg W = n+3$. Then $T \cong k[s,s^nt,s^nt^2,s^nt^3]$ where $s,t$ are indeterminates over $k$. Hence $T$ is an integral domain, and $T$ is a normal ring if and only if $n = 1$. The graded canonical module $\rmK_T$ of $T$ has the presentation of the form
\begin{equation}
  S(-(n+3))
  \oplus
  S(-(n+4))
  \oplus
  S(-(n+5))
\xrightarrow{\left(\begin{smallmatrix}X^n & Y & Z\\Y & Z & W\end{smallmatrix}\right)} 
  S(-3)\oplus S(-2)
\xrightarrow{\varepsilon} K_R\rightarrow 0. 
\end{equation} Since $$\mathrm{K}_T/T{\cdot}\varepsilon(\mathbf{e}_1) \cong [S/(Y,Z,W)](-2)\ \ \text{and} \ \ \mathrm{K}_T/T{\cdot}\varepsilon(\mathbf{e}_2) \cong [S/(X^n,Y,Z)](-3)$$ where $\mathbf{e}_1, \mathbf{e}_2$ is the standard basis of $S(-3)\oplus S(-2)$, $R$ is an almost Gorenstein local ring with $\mathrm{r} (R) = 2$. By Theorem  \ref{3.4} (1) every $I \in \calX_R$  contains $(x^n,y,z,w)R$, so that  $I = (x^\ell, y,z,w)R$ with $1 \le \ell \le n$. It is straightforward to check that $(x^\ell, y,z,w)R$ is actually an Ulrich ideal of $R$ for every $1 \le \ell \le n$.
\end{proof}

Let $k$ be a field and let $T =  k[X^n,X^{n-1}Y, \ldots, XY^{n-1},Y^n]$ be the Veronesean subring of the polynomial ring $S=k[X,Y]$ of degree $n \ge 3$. Let $R = T_M$ and $\m = MT_M$, where $M$ denotes the graded maximal ideal of $T$. Then $R$ is a non-Gorenstein almost Gorenstein normal local ring (\cite[Example 10.8]{GTT}) and we have the following. Let us note a brief proof in our context.

\begin{ex}[cf. {\cite[Example 7.3]{GOTWY2}}]\label{3.6}
$\calX_R = \{\m\}$.
\end{ex}

\begin{proof} Since $M^2=(X^n,Y^n)M$, we have $\m \in \calX_R$. Because $$\mathrm{K}_R = (X^{n-1}Y, X^{n-2}Y^2, \ldots, X^2Y^{n-2},XY^{n-1})R,$$ it is direct to check that 
$M \subseteq \sum_{i = 1}^{n-1}\left[R{\cdot}X^iY^{n-i}:_R \mathrm{K}_R\right].$ Hence $\calX_R = \{\m\}$ by Theorem \ref{3.4} (1).
\end{proof}

Let $S = k[[X,Y,Z]]$ be the formal power series ring over an infinite field $k$. We choose an element $f \in (X,Y,Z)^2 \setminus (X,Y, Z)^3$ and set $R = S/(f)$. Then $R$ is a two-dimensional Cohen--Macaulay local ring of multiplicity $2$. Let $\m$ denote the maximal ideal of $R$ and consider the Rees algebra $\calR = \calR(\m^\ell)$ of $\m^\ell$ with $\ell \ge 1$. Hence 
$$\calR=R[\m^\ell {\cdot} t] \subseteq R[t]$$ where $t$ is an indeterminate over $R$. Let $\fkM = \m + \calR_+$ and  set $A = \calR_\fkM$, $\n = \fkM \calR_\fkM$. Then $A$ is not a Gorenstein ring but almost Gorenstein (\cite[Example 2.4]{GTY1}). We furthermore have the following.

\begin{thm}\label{3.8}
$\calX_A=\{\n\}$.
\end{thm}

\begin{proof} We have $\m^2 = (a,b)\m$ with $a, b \in \m$. Let $Q = (a, b-a^\ell t, b^\ell t)$. Then $Q \subseteq \fkM$ and $\fkM^2 = Q\fkM$, so that  $\n \in \calX_A$. Conversely, let $J \in \calX_A$ and set $K = J \cap \calR$. We put $I = \m^\ell$ and choose elements $x_1, x_2, \ldots, x_q \in I$~($q := \nu (I) = 2\ell +1$) so that the ideal $(x_i, x_j)$ of $R$  is a reduction of $I$ for each pair $(i,j)$ with  $1 \le i  < j \le q$. (Hence $x_i, x_j$ form a super-regular sequence with respect to $I$, because $\gr_I(R)=\bigoplus_{n \ge 0}I^n/I^{n+1}$ is a Cohen--Macaulay ring.) Then   $$\sum_{i=1}^{q}\left[x_i\calR:_\calR I\calR\right] \subseteq K$$ by Corollary \ref{2.5},  since $\rmK_\calR(1) \cong I\calR$ (\cite[Proposition 2.1]{GTY1}). We have for each $1 \le i \le q$ that $$x_i\calR:_\calR I\calR = \sum_{n \ge 0}\left(x_iI^{n-1}\right)t^n,$$ because each pair $x_i, x_j~(i \ne j)$ forms a super-regular sequence with respect to $I$. Consequently, $$(x_i, x_it)\calR \subseteq K~ \text{for~all}~  1 \le i \le q.$$ Thus  $I + \calR_+ \subseteq K$ and hence $K = \fkb + \calR_+$ for some $\m$-primary ideal $\fkb$ of $R$. Notice that $[K/K^2]_1=I/\fkb I$ is $R/\fkb$-free, since $K/K^2$ is $\calR/K$- free and $\calR/K =[\calR/K]_0 = R/\fkb$. We then have $\fkb = \m$. In fact, let us write $\m = (a, b, c)$ so that each two of $a, b, c$ generate a reduction of $\m$. Set $\q = (a,b)$. Then since $\q$ is a minimal reduction of $\m$, the elements  $\{a^{\ell - i}b^i\}_{0 \le i \le \ell}$  form a part of a minimal basis, say  $\{a^{\ell - i}b^i\}_{0 \le i \le \ell}$ and $\{c_i\}_{1 \le i \le \ell}$, of $I= \m^\ell$. Let $\{\mathbf{e}_i\}_{0 \le i \le 2\ell}$ be the standard basis of $R^{\oplus (2\ell + 1)}$ and let
$$\varphi : R^{\oplus(2\ell+1)}=\bigoplus_{i=0}^{2\ell}R\mathbf{e}_i \to \m^{\ell}$$ be the $R$-linear map defined by $\varphi(\mathbf{e}_i) = a^{\ell-i}b^i$ for $0\le i \le \ell$ and $\varphi(\mathbf{e}_{i+\ell})  = c_{i}$ for $1 \le i \le \ell$.  Then setting  $Z = \Ker \varphi$, we get $\xi = \left[\begin{smallmatrix}
b\\
-a\\
0\\
\vdots\\
0
\end{smallmatrix}\right] \in Z$ and $\xi \not\in \m Z$ because $\xi \not\in \m^2{\cdot}R^{\oplus (2\ell + 1)}$. Hence $b, -a \in  \fkb$, because $I/\fkb I$ is $R/\fkb$-free. We similarly have $c, -a \in \fkb$. Hence  $\fkb = \m$, so that $K= \m + \calR_+ = \fkM$. Thus $J = \n$.
\end{proof}

\section{A method of constructing Ulrich ideals with different number of generators}\label{method}

This section purposes to show that in general the numbers of generators of Ulrich ideals are not necessarily constant. To begin with, we note the following.

\begin{lem}\label{4.1}
Let $\varphi : (A,\fkn) \to (R,\m)$ be a flat local homomorphism of Cohen--Macaulay local rings of the same dimension. Let $\q$ be a parameter ideal of $A$ and assume that $\n^2 =\q \n$. Then $J = \fkn R$ is an Ulrich ideal of $R$.
\end{lem}

\begin{proof} We set $Q = \q R$. Then $$J/Q \cong R \otimes_A(\fkn/\q) \cong R \otimes_A(A/\fkn)^{\oplus t}$$  where $ t = v(A) - \dim A \ge 0$. Hence $J$ is an Ulrich ideal of $R$, as $J^2 = QJ$.
\end{proof}

When $\dim R = 0$ and $R$ contains a field, every Ulrich ideal of $R$ is obtained as in Lemma \ref{4.1}.

\begin{prop}\label{prop2} 
Let $(R,\fkm)$ be an Artinian local ring which contains a coefficient filed $k$. Let $I=(x_1, x_2, \ldots, x_n)~(n = \nu(I))$ be an Ulrich ideal of $R$. We set $A = k[x_1, x_2, \ldots, x_n] \subseteq R$ and $\fkn = (x_1, x_2, \ldots, x_n)A$. Then $\n$ is the maximal ideal of $A$, $I = \fkn R$,  and $R$ is a finitely generated free $A$-module.
\end{prop}

\begin{proof} Let $r = \dim_kR/I$. Hence $r = \nu_A(R)$, as $I = \fkn R$. Let $\varphi : A^{\oplus r} \to R$ be an epimorphism of $A$-modules. Then $(A/\n) \otimes_A\varphi : (A/\n)^{\oplus r} \to R/I$ is an isomorphism, so that the induced  epimorphism $\phi : \fkn^{\oplus r} \to I$ must be an isomorphism, because $\dim_k\fkn^{\oplus r} = nr$ (remember that $\nu_A(\n) = n$) and $\dim_kI = \dim_k(R/I)^{\oplus n} = nr$. Hence $\varphi : A^{\oplus r} \to R$ is an isomorphism.
\end{proof}

When $\dim R > 0$, the situation is more complicated, as we see in the following.

\begin{ex} Let $V = k[[t]]$ be the formal power series ring over a field $k$ and set $R = k[[t^4,t^5]]$ in $V$. Then $(t^4,t^{10}), (t^8, t^{10})$ are Ulrich ideals of $R$. We set $A = k[[t^4, t^{10}]]$ and $B = k[[t^8, t^{10}, t^{12}, t^{14}]]$. Let $\m_A$ and $\m_B$ be the maximal ideals of $A$ and $B$, respectively. Then $A$ and $B$ are of minimal multiplicity with $(t^4, t^{10}) = \m_AR$ and $(t^8, t^{10}) = \m_BR$. Notice that $R \cong A^{\oplus 2}$ as an $A$-module, while $R$ is not a free $B$-module. We actually have $\rank_BR = 2$ and $\nu_B(R) = 4$. 
\end{ex}

Let $(A,\n)$ be a Noetherian local ring and $S = A[X_1, X_2, \ldots, X_\ell]$~($\ell > 0$) the polynomial ring. We choose elements $a_1, a_2, \ldots, a_\ell \in A$ and set $$\fkc = (X_i^2 -a_i \mid 1 \le i \le \ell) + (X_iX_j \mid 1 \le i,j \le \ell~\text{such~that}~i \ne j).$$ We put $R = S/\fkc$. 
Then $R$ is a finitely generated free $A$-module of $\rank_AR = \ell + 1$. We get $R = A{\cdot}1 + \sum_{i=1}^\ell A{\cdot}x_i$, where $x_i$ denotes the image of $X_i$ in $R$.  With this notation we readily get the following. We note a brief proof.

\begin{lem}\label{4.2}
Suppose that $a_1, a_2, \ldots, a_\ell \in \n$. Then $R$ is a local ring with maximal ideal $\m = \n R + (x_i \mid 1 \le i \le \ell)$.
\end{lem}

\begin{proof}
Let $M \in \operatorname{Max} R$ and $1 \le i \le \ell$. Then $a_i \in M$ since $a_i \in \n = M\cap A$, while $x_i \in M$ since $x_i^2 = a_i$. Thus $\n R + (x_i \mid 1 \le i \le \ell) \subseteq M$. Hence we get the result, because $\n R + (x_i \mid 1 \le i \le \ell) \in \operatorname{Max} R$.
\end{proof}

The following Theorem \ref{4.3} and Lemma  \ref{4.1} give a simple method of constructing of Ulrich ideals with different numbers of generators. In fact, suppose that $A$ has maximal embedding dimension and let $\q$ be a parameter ideal of $A$ such that $\n^2 =\q \n$. Then the ideals $I$ in Theorem  \ref{4.3} and $J$ in Lemma  \ref{4.1} are both Ulrich ideals of $R$  but the numbers of generators are different, if one takes the integer $\ell \ge 1$ so that $\ell \ne \e_\n^0(A)-1$.

\begin{thm}\label{4.3} 
Let $\q$ be a parameter ideal of $A$ and assume that 
\begin{enumerate}[$(1)$]
\item $A$ is a Cohen--Macaulay ring of dimension $d$ and 
\item $a_i \in \q^2$ for all $1 \le i \le \ell$. 
\end{enumerate}
Let  $I = \q R + (x_i \mid 1 \le  i \le \ell)$. Then $I$  is an Ulrich ideal of $R$ with $\nu (I) = d + \ell$.
\end{thm}

\begin{proof}
We set $Q = \q R$, $\fka = (a_i \mid 1 \le i \le \ell)$, and $\fkb = (x_i \mid 1 \le i \le \ell)$. Then $Q$ is a parameter ideal of $R$ and $I = Q + \fkb$. Therefore $I^2 = QI$ since $\fkb^2 = \fka \subseteq \q^2$. We set $m = \ell_A(A/\q)$. Hence $\ell_A(R/I) \le m$, as $R/I$ is a homomorphic image of $A/\q$. We consider the epimorphism $$(R/I)^{\oplus \ell} \overset{\varphi}{\longrightarrow} I/Q \to 0$$
of $R$-modules defined by $\varphi(\mathbf{e}_i) = \overline{x_i}$ for each $1 \le i \le \ell$,  where $\{\mathbf{e}_i\}_{1 \le i \le \ell}$ denotes the standard basis of $(R/I)^{\oplus \ell}$ and $\overline{x_i}$ denotes the image of $x_i$ in $I/Q$. Then $\varphi$ is an isomorphism, since 
$$
\ell_A(I/Q)
=\ell_A(R/Q) - \ell_A(R/I)
\ge(\ell +1) m - m
=\ell m
\ge\ell_A((R/I)^{\oplus \ell}).
$$
Thus $I$ is an Ulrich ideal of $R$ with $\nu (I) = d + \ell$. 
\end{proof}

\begin{ex}\label{4.5}
Let $0 < a_1 < a_2 < \cdots < a_n$~($n \ge 3$) be integers such that $\operatorname{gcd}(a_1, a_2, \ldots, a_n) = 1$. Let $L =\left<a_1, a_2, \ldots, a_n\right>$ be the numerical semigroup generated by $a_1, \a_2, \ldots, a_n$. Then $c(L) = a_n - a_1 + 1$, where $c(L)$ denotes the conductor of $L$. We set $A = k[[t^{a_i} \mid 1 \le i \le n]]$ ($k$ a field)  and assume that $v(A) = \e^0_{\n}(A)= n$, where $\n$ denotes the maximal ideal of $A$ (hence $a_1 = n$). We choose an odd integer $b \in L$ so that $b \ge a_n + a_1 + 1$ and consider the semigroup ring $R = k[[\{t^{2a_i}\}_{1 \le i \le n}, t^b]]$ of the numerical semigroup $H=2L +\left<b\right>$. Let $\varphi: A \to R$ be the homomorphism of $k$-algebras such that $\varphi (t^{a_i}) = t^{2a_i}$ for each $1 \le i \le n$.  We set $f = t^b$. Then $f^2 = \varphi (f) \in \varphi (A)$ but $f \not\in \varphi (A)$, since $b \in L$ is odd. Therefore $R$ is a finitely generated free $A$-module with  $\rank_AR = 2$ and $R = A{\cdot}1 + A{\cdot}f$. Since $f \in t^{2a_1}A$ by the choice of the integer $b$, we get by Theorem \ref{4.3} and Lemma  \ref{4.1} that $I = (t^{2a_1}, t^b)$ and $J = (t^{2a_i} \mid 1 \le i \le n)$ are both Ulrich ideals of $R$ with $\nu (I) = 2$ and $\nu (J) = n > 2$. 
\end{ex}

\section*{Acknowledgments}

The authors are grateful to Naoyuki Matsuoka for his helpful discussions to find Example \ref{4.5}. The authors are grateful also to the referee. Corollary \ref{referee} is suggested by the referee.


\end{document}